\documentclass[12pt]{article}

\usepackage{amsmath,makeidx,amssymb,amscd,amsfonts, amsthm}
\usepackage{graphicx,epsfig,latexsym,bbm,index}

\usepackage{chngcntr}

\usepackage[active]{srcltx}
\usepackage[below]{placeins}

\usepackage{psfrag}
\usepackage{color,epsfig}

\usepackage[dvips, bookmarksopen, colorlinks,
  linkcolor =blue, urlcolor =red, citecolor =red, menucolor =blue]{hyperref}

\usepackage[below]{placeins}
\newcommand\R{\mathbb{R}}

\newcommand\N{\mathbb{N}}

\newcommand\norm[2][]{{\left\lVert#2\right\rVert_{#1}}}

\newtheorem{theo}{Theorem}
\newtheorem{lemma}[theo]{Lemma}
\newtheorem{coro}[theo]{Corollary}
\newtheorem{pr}[theo]{Proposition}
\newtheorem{df}{Definition}
\newtheorem{remark}{Remark}

\def\ve{{\varepsilon}}

\def\D{{\cal D}}

\def\F{{\mathbb F}}

\def\.{{\;}}

\begin{document}
\setcounter{footnote}{1}

\title{Uniqueness and regularization for unknown spacewise lower-order coefficient and source for the heat type equation}

\author{
A.~De Cezaro$^\dagger$ %
%%\thanks{Institute of Mathematics Statistics and Physics,
%%Federal University of Rio Grande, Av. Italia km 8, 96201-900 Rio
%%Grande, Brazil (\href{mailto:decezaro@impa.br}{\tt
%%decezaro@impa.br}).}
\ \ and \ \
F. Travessini De Cezaro%
\thanks{Institute of Mathematics Statistics and Physics,
Federal University of Rio Grande, Av. Italia km 8, 96201-900 Rio
Grande, Brazil (\href{mailto:adrianocezaro@furg.br}{\tt
adrianocezaro@furg.br}).} }

\date{\small\today}

\maketitle
\begin{small}
\noindent {\bf Abstract:}

In this contribution we show  sufficient conditions for simultaneous
unique identification of unknown spacewise coefficients and heat
source in a parabolic partial differential equation given additional
final time measurements. Our approach is based on density, in
suitable spaces, of the corresponding adjoint problem.

A second issue of this paper is the regularization approach. The
sequence of approximated solution is obtained by coupling the
nonlinear Landweber iteration with iterated Tikhonov regularization.
We show that the parameter-to-solution map satisfies sufficient
conditions to prove stability and convergence of approximated
solutions for the identification problem. We use a unified
discrepancy principle as the stopping criteria.

Finally, we apply the developed theory in the inverse identification
problem of unknown parameters (perfusion coefficient, metabolic heat
source)  for the identification of tumor regions by thermography.
\vskip 0.5 cm %%
Keywords: uniqueness, thermophysical parameters and source
identification, iterative regularization, parabolic type equation,
final time measurements.
\end{small}
%%
%%
%%%----------------------------------------------------------------------------%
\section{Introduction} \label{sec:1}

Coefficients identification inverse problems have the characteristic
of being ill-posed \cite{EngHanNeu96}. In other words, typically,
solutions for such problems may fail to exist, may not be unique or
be unstable under errors in the input data. The issue of existence
can be relaxed by considering generalized solutions. On the other
hand, uniqueness and stability are crucial for obtaining a
reasonable solution for the coefficient identification inverse
problem, theoretically as well as in terms of numerical
approximations.%%
%%Therefore, and important issue in inverse problems and applications
%%is shows uniqueness and stability of solutions.

Uniqueness in inverse problems have been studied for a long time. In
particular, for coefficient identification in parabolic equations,
see \cite{Yamamoto2009, Isa06} and references therein. However, only
recently some attention has been given to the uniqueness of
coefficient identification in parabolic type equation with final
time measurements \cite{CezaroJohansson12, Goldman, Isa06,
Yamamoto2009}. A method that became very popular recently to prove
uniqueness and conditional stability in coefficient identification
inverse problems are Carleman type estimates, e.g. \cite{Isa06,
KlibanovTimonov2004, Yamamoto2009} and references therein. In
particular in \cite{Yamamoto2009} there is a complete overview of
Carleman estimates for parameter identification of inverse parabolic
problems. Recently, in \cite{CezaroJohansson12}, the authors
proposed using Carleman type estimates for identification of
spacewise source and heat conductivity in a heat type equation for a
given additional final time measurement. For the one-dimensional
heat equation, in \cite{CezaroJohansson12}, density properties of
the associated adjoint problem, in suitable spaces and an additional
assumption on the final measurements was used to guarantee
uniqueness identification of spacewise source and heat conductivity.
The duality type method was used before in \cite{Goldman11, Goldman}
for proving uniqueness identification of spacewise coefficients in
parabolic type equations. However, in \cite{Goldman11, Goldman} the
author does not consider simultaneous identification of coefficients
and heat source.

Likewise uniqueness, regularization approaches for parameter
identification in parabolic partial differential equations also have
a long history and a full overview becomes almost impossible. See,
for example, \cite{LALP2003, Ivanchov2003, EngHanNeu96, Isa06} and
references therein. We remark that, in general, parameter
identification inverse problems are nonlinear, even if the forward
problem is linear. Therefore, to prove convergence and stability of
iterative regularization methods in this context, ones need to prove
some nonlinearities conditions for the parameter-to-solution map
\cite{KaltNeuScher08}, that are, in general, hard to be verified in
practice.

\paragraph{Summarizing, the main contributions of this paper
are:}

We first use the density of the associated adjoint problem to prove
sufficient conditions for the simultaneous uniqueness identification
of spacewise heat source and coefficients that multiply the lower
order term (or lower-order derivative, see Remark~\ref{remark:2}) of
a parabolic type equation, with some assumption in the given final
data (see equation~\eqref{eq:ad-measurement} below). It is worth
noting that we do not have restrictions of space dimension, except
that necessaries to prove existence and uniqueness of a solution of
partial differential equation. Such restrictions are related to the
smoothness of the initial and boundary conditions and the smoothness
of coefficients of the partial differential equation.

We also propose a iterative regularization method that consists in
to couple Landweber and iterated Tikhonov regularization strategies.
We prove properties of the parameter-to-solution maps (see
definitions \eqref{operator_A} - \eqref{operator_a}) that are
sufficient conditions to show convergence and stability of
regularized solutions with respect to the noise level. We use a
stopping criteria given by a unified discrepancy principle. Such
discrepancy principle has the characteristic of reduce computational
effort (See Remark~\ref{remark:3}).

Finally, we apply the theory developed before to prove uniqueness
and provide and regularization approach for the identification of
the blood perfusion rate and the metabolic heat generation in a
thermography application for melanoma diagnoses.

\vskip 0.5 cm

The paper is organized as follows: In the remaining part of this
section we introduce some notations. In
Section~\ref{sec:model-problem}, we present the model problem and
define the parameter-to-solution map associated with the
identification problem. In Section~\ref{sec:uniqueness}, we prove
the uniqueness identification of the spacewise pair of coefficient
and source from the additional final time measurement in the
parabolic partial differential equation model. In
Section~\ref{section:Parameter-to-Solution}, we prove properties of
the parameter-to-solution map that guarantees the convergence and
stability of the iterative regularized solutions, that will be
proposed in Section~\ref{section:iterative}. In
Section~\ref{sec:numerics}, we present an application of the theory
developed early in a melanoma diagnoses from thermography. Finally,
in Section~\ref{sec:conclusion}, we present some conclusions and
future works.

\paragraph{Notations:}%%\label{sec:notation}

By $L^p(\Omega)$ for $1\leq p < \infty$, we denote the usual space
of $p$-integrable functions on $\Omega$ with the usual norm
$\norm{\cdot}_{L^p(\Omega)}$. The space $L^\infty(\Omega)$ is the
standard $L^\infty$-space. We denote by $W^{k,p}(\Omega)$ the
standard Sobolev space on $\Omega$ with generalized derivatives of
order $\leq k$ in $L^p(\Omega)$. In particular, for $p=2$ we have
the Hilbert spaces $H^k(\Omega)$.

Let $T > 0$ be fixed. Define the measurable function $u(\cdot,
t)\,:\; (0,T) \longrightarrow X$, where $X$ is a Banach space. We
denote by $C([0,T]; X)$ the space of continuous mappings $u(\cdot,
t)$ with the usual norm and by $L^2((0,T);X)$ the space of functions
such that
\begin{align*}
\int_0^T \norm{u(\cdot, t)}^2_{X} dt < \infty\,.
\end{align*}

\section{Model Problem}\label{sec:model-problem}

In this paper, we are considering a thermal-physical model in a
non-homogeneous and non-isotropic body, denoted by $\Omega$
occupying an open, bounded and smooth domain in $\R^n$, described by
the parabolic type partial differential equation
%%%
\begin{align}\label{eq:model2}
u_t - L(a,b,c)u  & = f(x) \,\, \mbox{in } \ \Omega \times (0,T)\nonumber\\
        u(x,t) & = 0     \,\, \mbox{for} \,\, (x, t) \,\, \in \partial \Omega \times (0,T)\\
        u(x,0) & = \varphi(x) \,\, \mbox{for }  x \in \Omega \,,\nonumber
\end{align}
for a time interval $(0,T)$ with $T>0$, where
\begin{align}\label{eq:L}
L(a,b,c)u = \nabla \cdot (a(x)  \nabla u) - b(x)\cdot \nabla u -
c(x) u\,,
\end{align}
is a linear elliptic differential operator of second order with all
the coefficients time independent. Moreover, $a,c$ are strictly
positive real valued function in $L^\infty(\Omega)$ with $0<
\underline{a} \leq a(x)$ and $0 < \underline{c} \leq c(x)$, for $x
\in \Omega$, $b$ is a real valued vector function sufficiently
smooth. $f \in L^2(\Omega)$ is the spacewise heat source. For
simplicity, we assume that the given initial temperature
distribution $\varphi \in H^2(\Omega) \cap H_0^1(\Omega)$, and that
$b=0$. Hence, we can assume that information that $u$ is identically
zero on the boundary of $\Omega$. We will denote $L(a,b,c) =
L(a,c)$.

Moreover, we assume given the additional final temperature
measurement $g\in H^2(\Omega) \cap H_0^1(\Omega)$ satisfying
\begin{align}\label{eq:ad-measurement}
u(x,T) =g(x) \,\, \mbox{for } x \in \Omega\,,\quad T > 0\,.
\end{align}

Since the parameters in \eqref{eq:model2} are sufficient smooth, we
can, formally, define the adjoint of the partial differential
operator $L(a,c)$ as
\begin{align}\label{eq:adjoint_L}
L(a,c)^*v =  \nabla \cdot (a(x)  \nabla v)  + \nabla (b(x)v) - c(x)
v\,.
\end{align}

Given the assumptions on the parameters in the
model~\eqref{eq:model2}, follows from the Hille-Yosida Theorem
\cite{Yosida95} that the operator $-L$ generate a strictly
dissipative and contraction $C_0$ semigroup $\{G(t)\}_{t \in
\mathbb{R}^+}$ in $L^2(\Omega)$ with $\mathcal{D}(L) = \{u\,:\, u
\in H_0^1(\Omega)\,,\, Lu \in L^2(\Omega)\}$. Hence, $\norm{G(T)}<
1$. Note that, since $H_0^1(\Omega)$ is compact embedded in
$L^2(\Omega)$, $G(t)$ is a compact operator for every $t> 0$.
%%%Moreover, $\mathcal{D}(L)$ is a dense set of its domain \cite{JL07}.
Moreover, it follows from classical results on parabolic partial
differential equations, e.g. \cite{LM-vol1, LandSolUr-1968} that
there exists a unique solution $u \in C^1((0,T), H_0^1(\Omega) \cap
H^2(\Omega))$ of \eqref{eq:model2}-\eqref{eq:ad-measurement} with
\begin{align}
\norm{u}_1:= \int_0^T\left( \norm{u_t(\cdot, t)}^2_{L^2(\Omega)} +
\norm{u(\cdot, t)}^2_{H^2(\Omega)}\right) dt < \infty.
\end{align}

\begin{remark}\label{remark:1}
\begin{itemize}
\item[i)] Is possible to extend the result to a weaker
assumption on the data, but one then has to consider appropriate
weak formulation of \eqref{eq:model2}-\eqref{eq:ad-measurement},
\cite{LandSolUr-1968}.

\item[ii)] Since we have time-independent coefficients and source, it follows
that the solution $u$ to \eqref{eq:model2} with additional data
\eqref{eq:ad-measurement} is analytic in time. This means that $u$
has derivatives of all orders with respect to $t$ \cite{Isa06,
Vessella2008}. Moreover $u$ is at least continuous with respect to
time, and, therefore, we can conclude that pointwise evaluation in
time makes sense.
\end{itemize}
\end{remark}

The inverse problem that we are interested here is recover the pair
of spacewise parameter and source $(c(x), f(x))$ in
\eqref{eq:model2} from the additional final time
measurement~\eqref{eq:ad-measurement}.

Assuming that $b=b(x)$ is not zero in \eqref{eq:model2}, one can
recover the pair of spacewise parameter and source $(b(x), f(x))$ in
\eqref{eq:model2} from the additional final time
measurement~\eqref{eq:ad-measurement}. However, the uniqueness
result, in this case, follows from a very similar argument (See
Remark~\ref{remark:2}) developed with this contribution and we will
not fix in this problem here.

\subsection{The parameter-to-solution map}\label{subsec:parameter-to-solution}

In this subsection, we introduce the parameter-to-solution map
associated with the parameter identification problem discussed
previously. For now, we will consider the following admissible set
of spacewise coefficient and heat source:
\begin{df}\label{df:admissible}
We denote the admissible set of functions as
\begin{align*}
\D(\F): = \{(c,f)  \in  L^\infty(\Omega) \times L^2(\Omega) \,
\mbox{ s.t. } 0 < \underline{c} \leq c(x) \leq \overline{c} \,,
\mbox{a.e. in\,}\, \Omega\}\,.
\end{align*}
Moreover we denote by $\mathcal{D}_c(\F) \subset L^\infty(\Omega)$
and $\mathcal{D}_f(\F)\subset L^2(\Omega) $ the restriction of
$\mathcal{D}(\F)$ to the first and second component of the pair $(c,
f)$, respectively.
\end{df}

Note that, since $\Omega$ is bounded, $\D(\F)$ is a convex and
closed subset of $L^2(\Omega) \times L^2(\Omega)$. However, since
$L^\infty(\Omega)$ can not be continuously embedding in
$H^1(\Omega)$ for $\Omega \subset \R^2$, $\D(\F)$ has no interior
points when equipped with the $H^1(\Omega)$ norm. This will not
affect the convergence analysis that follows. See
Section~\ref{section:Parameter-to-Solution} for details.

Moreover, let $u \in C([0,T]; H^1_0(\Omega) \cap H^2(\Omega))$ be
the unique solution of \eqref{eq:model2}, with $(c,f) \in
\mathcal{D}(\F)$. Then, it follows that the restriction $u(x,t_0)$
is well-defined for $0\leq t_0 \leq T$. Therefore, the restriction
of $u$ to $\Omega \times \{T\}$ exists, it is, $u(x,T) = g(x)$ is
well defined and, moreover, the nonlinear operators
\begin{align}\label{operator_A}
\F \,:\, & \mathcal{D}(\F) \subset L^2(\Omega) \times L^2(\Omega) \longrightarrow L^2(\Omega)\nonumber\\
       &  \quad  (\,c \,, \,f\,) \quad \longmapsto \quad  \F(c,f) = g(x)
\end{align}
is well-defined.

A second operator equation that we are interested in is the
restriction of the operator $\F$ defined in \eqref{operator_A}  in
$\D_f(\F)$. Indeed, it introduces a family of operators,
parameterized by $c \in \D(\F)$ defined by
\begin{align}\label{operator_F}
F_c \,:\, & \mathcal{D}_f(A) \subset L^2(\Omega) \longrightarrow L^2(\Omega)\nonumber\\
       &  \quad  f \longmapsto F_c(f) := \F_{|_{\D_f(\F)}}(c,f) =  g(x)\,.
\end{align}
Note that, for any fixed $c \in \D_f(\F)$, the operator  $F_c$ is
nonlinear, unless $\varphi=0$.

Finally, we introduce the restriction of the operator $\F$ to
$\D_c(\F)$ define a family of nonlinear operators, parameterized by
$f \in L^2(\Omega)$, defined by
\begin{align}\label{operator_a}
A_f \,:\, & \mathcal{D}_c(\F) \subset L^2(\Omega) \longrightarrow L^2(\Omega)\nonumber\\
       & \quad   c \longmapsto A_f(c) := \F_{|_{\D_c(\F)}}(c,f) =  g(x)\,.
\end{align}
In practical applications the final temperature
\eqref{eq:ad-measurement} is, in general, not known exactly. One is
given only approximate measured data $g^\delta \in L^2(\Omega)$,
corrupted by a noise level $\delta > 0$ satisfying
\begin{align}\label{eq:noise}
\norm{g - g^\delta}_{L^2(\Omega)} \leq \delta\,.
\end{align}
Therefore, given the measurements $g^\delta$, the inverse problem
is: Find $(c(x), f(x)) \in \D(\F)$ such that
\begin{align}\label{eq:operator}
\F(c,f) = g^\delta\,,\quad \mbox{ for } g^\delta \mbox{ satisfying
\eqref{eq:noise}\,.}
\end{align}
The inverse problem is ill-posed in the sense of Hadamard
\cite{EngHanNeu96}. In other words  the solution of the inverse
problem is unstable with respect to noise data, it is, small
perturbation in the given data implies in large perturbation on the
parameter space. Examples of instability of reconstruction the
coefficient $c$ can be obtained in \cite{Goldman11}. For
instabilities examples in the reconstruction of $f$ see \cite{JL08,
Rundel80}. Hence, beyond uniqueness, some regularization methods
need to be used to guarantee stability of the parameter and source
reconstructions, given a set of noisy data.

\section{Uniqueness for the spacewise coefficient and source}\label{sec:uniqueness}

In this section, we wish to prove uniqueness of the spacewise
coefficient $c$ and  source $f$ in \eqref{eq:model2} by additional
measurement~\eqref{eq:ad-measurement}. We will use a similar
approach of \cite[Section 3]{CezaroJohansson12} for the
identification of the heat conductivity and heat source for the
one-dimensional version of the heat equation~\eqref{eq:model2}. The
derivation of the uniqueness result is based on a completely
different technique than Carleman estimates \cite{CezaroJohansson12,
Isa06, Yamamoto2009}. Indeed, the technique is based on results of
density, in certain function spaces, of solutions of the
corresponding adjoint problem \cite{Goldman11, Goldman} and the
unique continuation principle \cite{LM-vol1, LandSolUr-1968}.
Moreover, the proposed approach is different of the maximum
principle used in \cite[Section 9.1]{Isa06}. It is worth remark
that, differently of \cite{Isa06, Goldman11, Yamamoto2009}, we wish
to show uniqueness of the spacewise coefficient $c$ and source $f$,
simultaneous, from~\eqref{eq:model2}-\eqref{eq:ad-measurement}.

The steps for proving uniqueness of the identification of the pair
of parameters $\{c(x), f(x)\}$ for giving initial and final data in
\eqref{eq:model2}-\eqref{eq:ad-measurement} are outlined as follows:

\paragraph{Preliminary results:} Denote by $u= u(c_1, f_1)$ and $v = u(c_2, f_2)$ the respective
solutions of \eqref{eq:model2} with additional data
\eqref{eq:ad-measurement}, for coefficient and source satisfying the
Definition~\ref{df:admissible}. Then, for linearity of
\eqref{eq:model2} the difference $ w = u - v$ satisfies
\begin{align}\label{eq:model2.1_1}
w_t - L(a,c_1) w = (c_1(x) - c_2(x)) v + (f_1(x) - f_2(x)) \,\,
\mbox{in } \, \Omega  \times (0,T)
%%        w(0,t)=u(L,t)& = 0     \,\, \mbox{for} \,\, t \in (0,T)\\
%%        w(x,0) & = 0 \,\, \mbox{for }  x \in (0,L) \,,\nonumber\\
%%        w(x,T) & = 0 \,\, \mbox{for }  x \in (0,L) \,,\nonumber
\end{align}
with homogeneous initial, boundary and final conditions.

Let we invoke the adjoint problem of \eqref{eq:model2.1_1}, that
reads as
%%%
\begin{align}\label{eq:model2.Adj}
\psi_t + L^*(a,c_1) \psi & = 0 \,\,
\mbox{in } \,\Omega  \times (0,T)\nonumber\\
        \psi(x,t)& = 0     \,\, \mbox{for} \,\, (x,t) \in \partial \Omega \times (0,T)\\
     \psi(x,0) & =  0 \,\, \mbox{for }  x \in \Omega
     \,,\nonumber\\
   \psi(x,T) & =  \mu(x) \,\, \mbox{for }  x \in \Omega,
\,,\nonumber
\end{align}
where $\mu(x)$ is an arbitrary function in
$C_0^2(\overline{\Omega})$.

For properties of solutions of the adjoint equation
\eqref{eq:model2.Adj} we have:

\begin{lemma}\label{lemma:prop-adj}
Let the Assumption on the coefficients, source, initial, boundary
and final conditions in \eqref{eq:model2}-\eqref{eq:ad-measurement}
hold. Then:
\begin{itemize}
\item[i)] For any function $\mu(x) \in C_0^2(\overline{\Omega})$, there exists a unique solution $\psi(x,t; \upsilon) \in
C^1((0,T); H_0^1(\Omega) \cap H^2(\Omega))$  of
\eqref{eq:model2.Adj}.

\item[ii)] For any function $\mu(x) \in C_0^2(\overline{\Omega})$, the following relation holds
\begin{align}\label{eq:1}
\int_0^T\int_\Omega \psi(x,t; \mu(x)) F(x,t) dxdt = 0,
\end{align}
where $w$ is a solution to \eqref{eq:model2.1_1} with right-hand
side  $F(x,t): = (c_1(x) - c_2(x)) v + (f_1(x) - f_2(x)) $.

\item[iii)]
For $\mu(x)$ ranging over the space $C_0^2(\overline{\Omega})$, the
corresponding range of $\psi(x,t; \mu(x))|_{t = \tau}$ is everywhere
dense in $L^2(\Omega)$ at any time $t=\tau$, $0\leq \tau \leq T$.

\item[iv)] Given that
$$
\int_0^T \int_\Omega \psi(x,t; \mu(x)) \Phi(x,t) dxdt = 0
$$
for $\mu(x)$ ranging over the space $C_0^2(\overline{\Omega})$, then
$ \Phi(x,T)=0\,,\, \mbox{a.e. } x\in \Omega\,. $

\end{itemize}
\end{lemma}

\begin{proof}
Item i) is a well-known result for parabolic equations, see, for
example, \cite{LM-vol1, LandSolUr-1968}. \\
Item ii) follows immediately  by multiplication of
\eqref{eq:model2.3} by $\psi$  and integration by parts.\\
Item iii) and Item iv) are consequences of \cite[Lemma 2-3, pg
318]{Goldman} applied for a multi-dimensional case.
\end{proof}

\begin{lemma}\label{lemma:auxiliar}
Let $w$ satisfying \eqref{eq:model2.1_1}. Then $w(x,t) = 0$ almost
everywhere for $(x,t) \in \overline{\Omega} \times [0,T]$.
\end{lemma}
\begin{proof}
Note that, since $w$ has homogeneous boundary conditions, it follows
for standard parabolic theory that $w(\cdot, t) \in H^1_0(\Omega)$,
for any $t\in (0,T)$. From Remark~\ref{remark:1}~ii) we have that
$w$ has derivatives of all orders with respect to $t$ and time
pointwise evaluation makes sense.

Applying Lemma~\ref{lemma:prop-adj} Item ii) in combination with iv)
in \eqref{eq:model2.1_1} we get
\begin{align}\label{eq:zero_1}
[(c_1(x) - c_2(x))v(x,T) + (f_1(x) + f_2(x))] = 0\,,\quad \mbox{a.e.
}\, x \in \Omega\,.
\end{align}
Using this in the first equation in (\ref{eq:model2.1_1}) in
combination with $w(x,T)=0$ for a.e. $x$ in $\Omega$, we conclude
that $w_t(x,T)=0$, a.e. $x$ in $\Omega$.

Define $z = w_t$. Since coefficients and source are time
independent, we have that $z$ satisfies
\begin{align}\label{eq:model2.3}
z_t - L(a,b,c_1) z& = (c_1(x) - c_2(x)) v_t \,\, \mbox{in } \ \Omega
\times (0,T)\nonumber\\
        z(x,t)& = 0     \,\, \mbox{for} \,\, (x,t) \in \partial \Omega \times (0,T)\\
     z(x,0) & =  \theta_1(x) \,\, \mbox{for }  x \in \Omega
     \,,\nonumber\\
     z(x,T) & =  0 \,\, \mbox{for }  x \in \Omega.
     \,\nonumber
\end{align}
Splitting this problem into two, one with zero right-hand side and
with  initial condition $\theta_1$ and one with the given right-hand
side and zero initial condition, following the proof of
\cite[Lemma~2]{CezaroJohansson12} one can conclude that the solution
to the first one is identically zero, i.e. $\theta_1(x)=0$.
Therefore, since $z$ satisfies a problem of the same kind as $w$ we
can again apply Lemma~\ref{lemma:prop-adj} Item ii) in combination
with iv) in \eqref{eq:model2.3} to conclude that
\begin{align}
(c_1(x) - c_2(x)) v_t(x,T) = 0\,,\quad \mbox{a.e. }\, x \in
\Omega\,.
\end{align}
Using this in the first equation in (\ref{eq:model2.3}) in
combination with $z(x,T)=0$ for a.e. $x$ in $\Omega$, we conclude
that $z_t(x,T)=0$, i.e. $w_{tt}(x,T)=0$, a.e. $x$ in $\Omega$.
Continuing this, putting $z_1=z_t$ and deriving the problem for
$z_1$ and applying the similar reasoning, i.e.
Lemma~\ref{lemma:prop-adj} Item ii) in combination with iv), we find
that $w_{ttt}(x,T)=0$. Further continuing this it is possible to
prove that $(\partial_t^{(k)} w)(x,T) = 0$ for $k=0,1,2,\ldots$.
From this and strong unique continuation results for parabolic
equations \cite{LM-vol1, LandSolUr-1968}, we conclude that
$w(x,t)=0$ for a.e. $(x,t) \in \Omega \times [0,T]$.
\end{proof}

\paragraph{The uniqueness proof:}

We now have the required results in order to prove the main step in
the uniqueness of $(c(x), f(x))$ in \eqref{eq:model2}, with
additional final data~\eqref{eq:ad-measurement}.
\begin{theo}\label{theo:uniqueness-a}
Let the Assumption on this paper holds. Moreover, assume that $g(x)
\neq  \varphi(x)$ for a.e $x \in \Omega$. Then the inverse problem
\eqref{eq:model2}-\eqref{eq:ad-measurement} has a unique solution
$\{u, c, f\}$ with the coefficient $c \in L^\infty(\Omega)$, the
heat source $f \in L^2(\Omega)$,  and temperature $u$, with
$\norm{u}_1 < \infty$.
\end{theo}
\begin{proof}
Follows from Lemma~\ref{lemma:auxiliar} that $w$ is identically zero
in $\Omega \times [0,T]$. This in particular implies that $z=0$ in
$\Omega \times [0,T]$.

From the first equation in (\ref{eq:model2.3}) we then have
$$
(c_1(x) - c_2(x))v_t(x,t) = 0\,, \mbox{for a.e. } (x,t) \in \Omega
\times [0,T]\,.
$$
Since the coefficients are independent of time, we integrating with
respect to time, form $0$ to $T$ and use the fundamental theorem of
calculus,  to get
$$
(c_1(x) - c_2(x))(g(x) - \varphi(x) ) = (c_1(x) - c_2(x))(v(x,T) -
v(x,0)) = 0\,, \mbox{for a.e. } x \in \Omega \,.
$$
From assumptions on $g$ and $\varphi$ and we can conclude that
$c_1(x)=c_2(x)$ also for a.e $x\in \Omega$. %%
%%The final step in  the uniqueness argument is the proof of unique
%%identifiability of $f(x)$ in
%%\eqref{eq:model2}-\eqref{eq:ad-measurement}.

Moreover, since $w=0$ (from Lemma~\ref{lemma:auxiliar}) and
$c_1(x)=c_2(x)$, we have from (\ref{eq:model2.1_1}) that
$$
f_1(x) - f_2(x) = 0\,, \mbox{for a.e. } x \in \Omega \,.
$$
The uniqueness of $u$, with $\norm{u}_1$ follows from the standard
theory of solution of parabolic partial differential equations
\cite{LandSolUr-1968}.
\end{proof}

It is worth to note that the argument in the proof of
Theorem~\ref{theo:uniqueness-a} goes beyond the proof of uniqueness
in \cite{CezaroJohansson12}. The main reason is that, for the lower
order terms we do not have the influence of the divergent operator.
Therefore the proof is still true in multidimensional heat
equations.

\begin{remark}\label{remark:2}
With similar argumentation of Theorem~\eqref{theo:uniqueness-a} is
possible to prove uniqueness of $\{u,b,f\}$ for the
model~\eqref{eq:model2}-\eqref{eq:ad-measurement}. Indeed, some
small modifications in the derivations of the steps above are
necessary. The main difference is that in
Theorem~\ref{theo:uniqueness-a} we need the assumption that $|\nabla
g(x) - \nabla \varphi(x)| > 0 $ for a.e $x \in \Omega$.
\end{remark}

%%%%%%%%%%%%%%%%%%%%%%%%%%%%%%%%%%%%%%%%%%%%%%%%%

\section{Properties of the Parameter-to-Solution
Map}\label{section:Parameter-to-Solution}

Before introducing the iterative regularization in
Section~\ref{section:iterative}, we need to prove some properties of
parameter-to-solution map defined before allowing us to obtain
convergence, stability and regularization properties of approximated
solutions.

\subsection{Continuity}\label{subsec:Continuity-P-t-S}

The first result in this direction is the continuity of operators
$\F$, $A_f$ and $F_c$ defined in \eqref{operator_A} -
\eqref{operator_F}, respectively.

\begin{theo}\label{theo:continuity}
The operator $\F : \mathcal{D}(\F) \subset L^2(\Omega) \times
L^2(\Omega) \longrightarrow L^2(\Omega)$, defined in
\eqref{operator_A} is continuous.
\end{theo}

\begin{proof}
Let $\{(c_n, f_n)\}$ be a sequence in $\mathcal{D}(\F)$ converging
to some $(c_0, f_0) \in \mathcal{D}(\F)$ w.r.t. $ L^2(\Omega) \times
L^2(\Omega)$-norm. Denote by $u_n = u(c_n, f_n)$ and $v = u(c_0,
f_0)$, respectively, the solutions of \eqref{eq:model2}. As before,
the difference $w := u_n - v$ satisfies
\begin{align*}
w_t - \nabla \cdot ( a(x) \nabla  w) + c_n(x) w =   (c_0 -c_n) v +
(f_n - f_0)
\end{align*}
with homogeneous, initial, boundary and final conditions.

Since $u_n, v \in L^2(0,T; H^1_0(\Omega)) \cap C([0,T];
H^1_0(\Omega))$, for each $t \in [0,T]$ we have
\begin{align*}
\int_\Omega w_t w - \nabla \cdot ( a_n \nabla w) w + c_n(x) w w   dx
= \int_\Omega (c_0 - c_n)v w dx + \int_\Omega (f_n - f_0) w dx
\end{align*}

The Green's formula \cite{LM-vol1}, implies that
\begin{align*}
\frac{1}{2}\frac{d}{dt} \norm{w(t)}_{L^2(\Omega)}^2 + \underline{a}
\norm{ \nabla w(t) }^2_{L^2(\Omega)} + \underline{c}
\norm{w(t)}^2_{L^2(\Omega)} \leq \int_\Omega |c_n - c_0|
|v(t)||w(t)| dx + \int_\Omega |f_n - f_0||w(t)|dx \,,
\end{align*}
where we used the homogenous initial and boundary conditions of the
equation satisfied by $w$. Moreover, from Theorem~\ref{th:Meyers},
there exists some $\tilde{\varepsilon}> 0$ such that $\norm{v(t)
}_{W^{1,q}(\Omega)} \leq C$, for $ q =2 + \tilde{\varepsilon}$.
Since the application $t \mapsto \norm{v(t) }_{W^{1,q}(\Omega)} $ is
continuous ($v \in C([0,T], W^{1,q}(\Omega))$), we have that is
uniformly bounded for $t\in [0,T]$.

Let $p^{-1} + q^{-1} = 2^{-1}$. Using the H\"{o}lder inequality with
$p^{-1} + q^{-1} + 2^{-1} = 1$, we have
\begin{align*}
\frac{1}{2}\frac{d}{dt} \norm{w(t)}_{L^2(\Omega)}^2 & +
\underline{a} \norm{ \nabla w (t) }^2_{L^2(\Omega)} + \underline{c}
\norm{w (t) }^2_{L^2(\Omega)} \\
& \leq \norm{v(t)}_{W^{1,q}(\Omega)}
\norm{c_n - c_0}_{L^p(\Omega)} \norm{w (t)}_{L^2(\Omega)} +
\norm{f_n -
f_0}_{L^2(\Omega)} \norm{w(t)}_{L^2(\Omega)}\\
&   \leq C \left(\norm{c_n - c_0}_{L^p(\Omega)} + \norm{f_n -
f_0}_{L^2(\Omega)} \right)\norm{w (t)}_{L^2(\Omega)} \,.
\end{align*}

Since, for each $t \in (0,T)$ $w \in H_0^1(\Omega)$, it follows from
the  Poincar\'e inequality that $ \norm{ w }^2_{L^2(\Omega)} \leq
C_1 \norm{ \nabla w }^2_{L^2(\Omega)}$. This, together with the
Young inequality with $\hat{\varepsilon}$ \cite{LM-vol1}, yields
\begin{align*}
min\{2^{-1}, C_1\underline{a}, \underline{c} \} & \left(\frac{d}{dt} \norm{w(t)}_{L^2(\Omega)}^2 +  \norm{w (t)}^2_{L^2(\Omega)} \right)\\
& \leq \frac{C}{\hat{\varepsilon}}\left( \norm{c_n -
c_0}_{L^p(\Omega)} + \norm{f_n - f_0}_{L^2(\Omega)} \right)^2
+ C\hat{\varepsilon} \norm{w(t)}^2_{L^2(\Omega)} \\
\end{align*}

Let $\hat{\varepsilon} > 0$ such that $C \hat{\varepsilon} >
\min\{2^{-1}, C_1\underline{a}, \underline{c}\}$. Using the Gronwall
inequality, it follows that
\begin{align*}
 \norm{w(t)}_{L^2(\Omega)}
& \leq \underline{C} \left( \norm{c_n - c_0}_{L^p(\Omega)} +
\norm{f_n - f_0}_{L^2(\Omega)} \right)^2 e^{ct}
\end{align*}
for all $t \in [0,T]$. Given the continuity of the solution of
\eqref{eq:model2} with respect to $t$ the inequality holds for
$t=T$.

Now the conclusion follows form Lemma~\ref{lemma:aux} in Appendix.
\end{proof}

As a corollary of Theorem~\ref{theo:continuity} we conclude the
continuity of operators $A_f$ and $F_c$, for any fixed $f$ and $c$
in $\D(\F)$, respectively.

\begin{coro}\label{coro-continuity}
For each fixed $c \in \D(\F)$, the operator $F_c$ defined in
\eqref{operator_F} is continuous in $L^2(\Omega)$.  For each fixed
$f \in \D(\F)$, the operator $A_f$ defined in \eqref{operator_a} is
continuous in $L^2(\Omega)$.
\end{coro}

\subsection{Fr\'echet Derivative and Tangential Cone Condition}\label{subsec:derivative}

An important result to guarantee convergence of iterative
regularization methods for nonlinear inverse problems is the local
tangential cone condition. We will prove such properties in the next
two propositions.

\begin{pr}\label{lemma:TCC_F}
For each fixed $c \in \D_c(\F)$ the operator $F_c$ is Fr\'echet
differentiable. The Fr\'echet derivative is Lipschitz continuous and
satisfies the local tangential cone condition. In other words, for
each $f, \tilde{f} \in B_\rho(f_0) \subset D_f(\F)$, there exists a
$0 <  \eta < 1$, such that
\begin{align}
\norm{F_c(\tilde{f}) - F_c(f) - F_c'(f)(\tilde{f} -
f)}_{L^2(\Omega)} \leq \eta \norm{F_c(\tilde{f}) -
F_c(f)}_{L^2(\Omega)}\,.
\end{align}
\end{pr}

\begin{proof}
Let $h \in L^2(\Omega)$.  By linearity of \eqref{eq:model2}, the
sensitivity $u'\cdot h = u(a,f+h) - u(f)$ satisfies
\begin{align}\label{eq:dirivative-F}
(u'\cdot h)_t - \nabla \cdot (a  \nabla (u'\cdot h)) + c(x) u'\cdot
h = h\,,
\end{align}
with homogeneous initial, final and boundary condition. Follows from
the standard parabolic theory that there exists a unique solution
$C([0,T], H^1_0(\Omega)\times H^2(\Omega))$ of
\eqref{eq:dirivative-F} and %%
$$
\norm{u'\cdot h}_{C([0,T], H^1(\Omega))} \leq C
\norm{h}_{L^2(\Omega)}\,.
$$

Let $\{G(t)\}_{t \in \mathbb{R}^+}$ be the semigroup generated by
the differential operator $-L(a,c)$. Define the linear operator
$K(t): L^2(\Omega) \longrightarrow L^2((0,T); H_0^1(\Omega)) \cap
C([0,T]; H_0^1(\Omega))$ by $K(t)f = \int_0^t G(t-s) f ds$. Note
that the solution $u$ of \eqref{eq:model2} is formally given by
$u(x,t) = G(t) g(x) + K(t) f(x)$. Therefore, solving the operator
equation \eqref{operator_F} is equivalent to solving $$ K(T)f =
u(c,x,T) - G(T)g(x)\,.$$
In other words, $(F'_c(f))(h) = K(T)h = u'(x,T)\cdot h$, where
$u'\cdot h$ is the unique solution of \eqref{eq:dirivative-F}.
Therefore, follows from the linearity and continuity of $K$ that
$F'_c$ is Lipschitz continuous and satisfies the tangential cone
condition.
\end{proof}

\begin{pr}\label{lemma:TCC_A}
For each fixed $f \in \D_f(\F)$, the operator $A_f$ is
differentiable in the direction $\kappa$ such that $c + \kappa \in
D_c(\F)$. The derivative can be continuously extended as a linear
operator to $H^1(\Omega)$. The extension is Lipschitz continuous.
Moreover, the local tangential cone condition is satisfied. In other
words, there exists $\rho > 0$ and $0 < \eta < 1$ such that for each
$c, \tilde{c} \in B_\rho(c_0) \subset D_c(\F)$,
\begin{align}
\norm{A_f(\tilde{c}) - A_f(c) - A_f'(c)(\tilde{c} -
c)}_{L^2(\Omega)} \leq \eta \norm{A_f(\tilde{c}) -
A_f(c)}_{L^2(\Omega)}\,.
\end{align}
\end{pr}
\begin{proof}
Let $\kappa \in D_c(\F)$. Then, by linearity  and continuity with
respect to the coefficients of solutions of
equation~\eqref{eq:model2} we have that the directional derivative
$u'_{(c)} \cdot (\kappa)$ in the direction $\kappa$ such that $c +
\kappa \in \D(\F)$ satisfies
\begin{align}\label{eq:Frechet-deriv-A}
(u'_{(c)}\cdot (\kappa))_t - L(a,c) u'_{(c)}\cdot (\kappa) = \kappa
u\,,
\end{align}
with homogeneous initial and boundary conditions. It follows from
standard parabolic partial differential equation theory
\cite{LM-vol1, LandSolUr-1968} that there exists a unique solution
$u'_{(c)} \cdot (\kappa) \in C^1((0,T); H_0^1(\Omega) \cap
H^2(\Omega)) \cap C^0([0,T]; L^2(\Omega))$ of
\eqref{eq:Frechet-deriv-A}. Hence, the $u'_{(c)}\cdot (\kappa)(t)$
make sense, for every $t \in [0,T]$. Moreover,
$$
\norm{u'_{(c)}\cdot (\kappa)}_{C^0([0,T], L^2(\Omega))} \leq
\norm{\kappa u}_{L^2(\Omega \times (0,T))}\,.
$$
Since $\kappa$ is time-independent, follows from the Cauchy-Schwarz
inequality that
$$
\norm{u'_{(c)}\cdot (\kappa)(T)}_{L^2(\Omega)} \leq C(T)
\norm{\kappa}_{H^1(\Omega)} \norm{u}_1\,.
$$
Therefore, the directional derivative $u'_{(c)} \cdot (\kappa)$ can
be extended to $H^1(\Omega)$ as a bounded linear operator.

To prove the Lipschitz continuity, let $c, \tilde{c} \in \D(\F)$ and
$u = u(c), \tilde{u} = u(\tilde{c})$ the respective solution of
equation~\eqref{eq:model2}. Then the difference $w = u'_{(c)} \cdot
(\kappa) - \tilde{u}'_{(\tilde{c})} \cdot (\kappa)$ satisfies
\begin{align*}
w_t - L(a,c) w = (c - \tilde{c}) \tilde{u} + \kappa (u -
\tilde{u})\,.
\end{align*}
As before, we have
$$ \norm{w(T)}_{L^2(\Omega)} \leq \norm{c- \tilde{c}}_{H^1(\Omega)} \norm{u}_1 +
\norm{\kappa}_{H^1(\Omega)} \norm{u - \tilde{u}}_{L^2((0,T)\times
\Omega)}\,.$$
With the same argumentations as in Theorem~\ref{theo:continuity}, we
have $\norm{u - \tilde{u}}_{L^2((0,T)\times \Omega)} \leq C \norm{c
- \tilde{c}}_{H^1(\Omega)}$ and the Lipschitz continuity follows.

Moreover,  from the linearity of equations~\eqref{eq:model2} and
\eqref{eq:Frechet-deriv-A} we get that $v = u(c) - u(\tilde{c}) -
u'(\tilde{c})\cdot (c-\tilde{c})$ satisfies
\begin{align}\label{eq:Frechet-deriv-A-tcc}
v_t - L(a, \tilde{c}) v = (c - \tilde{c}) ( u(c) - u(\tilde{c}))\,,
\end{align}
with homogeneous boundary and initial conditions. Using similar
argumentation as in the proof of Theorem~\ref{theo:continuity}, we
obtain that
\begin{align*}
\frac{1}{2}\frac{d}{dt} \norm{v(t)}^2_{L^2(\Omega)} &  +
\underline{a}\norm{\nabla v(t)}^2_{L^2(\Omega)} +
\underline{c}\norm{v(t)}^2_{L^2(\Omega)} \leq \int_\Omega (c -
\tilde{c}) (u(c) - u(\tilde{c})) v(t) dx
\\ & \leq \frac{ \norm{c -
\tilde{c}}_{L^\infty(\Omega)}}{\varepsilon}\norm{u(c) -
u(\tilde{c})}^2_{L^2(\Omega)} + \varepsilon
\norm{v(t)}^2_{L^2(\Omega)}\,,
\end{align*}
where we use the Young inequality with $\varepsilon$.

Let we take $\varepsilon < \underline{c}/2$. Therefore, we have
$$
\norm{v(t)}_{L^2(\Omega)} \leq C \norm{u(c) - u(\tilde{c})
}_{L^2(\Omega)}\,, \quad \forall t\in (0,T)
$$
where $C = C(\norm{c - \tilde{c})}$ and some constants that are
independent of the solution of \eqref{eq:model2}.
Denoting $\eta: = C(\norm{c - \tilde{c}})$, we have that
$$
\norm{v(t)}_{L^2(\Omega)} \leq \eta \norm{u(c) - u(\tilde{c})
}_{L^2(\Omega)}\,, \quad \forall t\in (0,T)\,.
$$
Given the continuity of $v$ and $u$ with respect to $t$, the
inequality holds for $t=T$.
\end{proof}

\paragraph{Adjoint of the Fr\'echet derivative:}
Let we finish this section making the calculation of the adjoint of
the Fr\'echet derivative of the operator defined in
\eqref{operator_F}.

\begin{lemma}\label{coro:adjoint-der-F}
Let $r_f \in L^2(\Omega)$. Then the adjoint of the Fr\'echet
derivative $F'_c(f)$ denoted by $(F'_c(f))^* : L^2(\Omega) \to
L^2(\Omega)$ is given by
\begin{align}\label{eq:adjoint-F}
(F'_c(f))^*(r_f) = - V(x,0)\,,
\end{align}
where $V \in C([0,T], H_0^1(\Omega))$ is the unique solution of
\begin{align}\label{eq:adjoint-equation-F}
V_t + \nabla \cdot (a \nabla V) + c V & = r_f\,, \quad \mbox{ in }
\,\,
(0,T) \times \Omega\,,\\
V(x,t) & =  0 \,,\quad (x,t) \in \partial \Omega \times (0,T) \nonumber\\
V(x,T) & = 0\,, \quad \mbox{for } x \in \Omega\,. \nonumber
\end{align}
\end{lemma}
\begin{proof}
Existence, uniqueness and regularity of the solution of
\eqref{eq:adjoint-equation-F} follows from standard parabolic theory
\cite{LandSolUr-1968}. Note that
equation~\eqref{eq:adjoint-equation-F} is the adjoint equation of
\eqref{eq:dirivative-F}. Therefore, the assertion follows directly
from integration by parts. See the details of the calculations in
\cite{JL08}.
\end{proof}
%
%%%
%%%
%%%\begin{lemma}\label{coro:adjoint-der-A}
%%%Let $r_c \in L^2(\Omega)$. Then the adjoint of the Fr\'echet
%%%derivative $A'_f(c)$ denoted by $(A'_f(c))^* : L^2(\Omega) \to
%%%L^2(\Omega)$ is given by
%%%%
%%%\begin{align}\label{eq:adjoint-A}
%%%(A'_f(c))^*(r_c) =  \langle u(t,\cdot) , W(t,\cdot)
%%%\rangle_{L^2(0,T)}\,, \mbox{suspeito que seja assim. tem que fazer
%%%as contas}
%%%\end{align}
%%%%
%%%where $W \in C([0,T], H_0^1(\Omega))$ is the unique solution of
%%%%
%%%\begin{align}\label{eq:adjoint-equation-A}
%%%W_t +  \nabla (a \nabla W) + c(x) W & = r_c\,, \quad \mbox{ in }
%%%\,\,
%%%(0,T) \times \Omega\,,\\
%%%W(x,t)  & =  0 \,, \quad (x,t) \in \partial \Omega \times (0,T) \nonumber\\
%%%W(T,x) & = 0\,,  \quad \mbox{for } x \in \Omega\,. \nonumber
%%%\end{align}
%%%%
%%%\end{lemma}
%%%%
%%%\begin{proof}
%%%Note that the Fr\'echet derivative of the operator $A_f$ satisfies
%%%\eqref{eq:Frechet-deriv-A}.  Therefore, using Green's formula
%%%\cite{Evans98}, we have
%%%%
%%%\begin{align*}
%%%\langle (A'_f(c))^*(r_c) , h \rangle_{L^2(\Omega)} & = \langle r_c ,
%%%A'_f(c)(h) \rangle_{L^2(\Omega)} = \int_0^T \int_\Omega (W_t + \nabla \cdot (a \nabla W) + c(x)W) u'(c)(h) dx dt\\
%%%&  = \int_0^T \int_\Omega ((u'(c)(h))_t - \nabla \cdot (a \nabla
%%%u'(c)(h)) + c u'(c)(h) ) W  dx
%%%dt = \int_0^T \int_\Omega - h u W dx dt \\
%%%&  =  \int_\Omega  h \left( \int_0^T  u W dt\right) dt dx\,.
%%%\end{align*}
%%%\end{proof}
%%%%

\section{An Iterative Regularization
Method}\label{section:iterative}

In this section, we propose an iterative regularization method to
regularize the solution of the inverse problem \eqref{operator_A}.
It consist in a coupled  Landweber - iterated Tikhonov
regularization approach given by the iteration
\begin{align}\label{eq:iteration}
\mbox{Given } c_0 = c_0^\delta,\, & f_0 = f_0^\delta \in \D(\F)\,,\quad \mbox{ for } k=0,\cdots, k_* \nonumber\\
f_{k+1}^\delta & = f_k^\delta + \gamma
(F'_{c_k^\delta}(f_{k}^\delta))^*(g^\delta -
F_{c_{k}^\delta}(f_{k}^\delta))\\
c_{k+1}^\delta   & \in \, argmin \,\, J_{\alpha}(c) : =
\norm{A_{f_{k+1}^\delta}(c) - g^\delta}^2_{L^2(\Omega)} + \alpha
\norm{c - c_{k}^\delta}^2_{H^1(\Omega)}\nonumber\,,
\end{align}
where, $k_*$  is the stopping index, determine by the stopping
criterion using the discrepancy principle
\begin{align}\label{eq:discrepancy}
\norm{\F(c_{k+1}^\delta, f_{k+1}^\delta) - g^\delta}_{L^2(\Omega)}
\leq \tau \delta < \norm{\F(c_{k}^\delta, f_{k}^\delta) -
g^\delta}_{L^2(\Omega)}\,,
\end{align}
and the relaxation parameter $\tau$ is such that
\begin{align}\label{eq:r}
\tau > 2 \frac{1+ \eta}{1 - \eta}\,.
\end{align}

Note that, if we have noise free data then $k_* = + \infty$. In this
case we drop the index $\delta$ in \eqref{eq:iteration}.

Moreover, in \eqref{eq:iteration} the positive parameter $\gamma $
is a scaling factor to enforce convergence of the Landweber
iteration \cite{KaltNeuScher08, EngHanNeu96}. As a consequence of
Lemma~\ref{coro:adjoint-der-F}, we have
\begin{align*}
f_{k+1}^\delta = f_k^\delta + \gamma V_k(x,T) = f_k^\delta + \gamma
K^*(T)(g^\delta - u_k(x,T))\,,
\end{align*}
where $V_k(x,t)$ is the unique solution of
\eqref{eq:adjoint-equation-F} with $r_f = u_k(x,T) - g^\delta$ and
$u_k(x,T) = u(c^\delta_k, f_k^\delta)$ is the unique solution of
\eqref{eq:model2} with $c = c^\delta_k$ and $f = f^\delta_k$,
respectively. Therefore, is enough that $\gamma < \norm{K(T)}^{-2}$.
Since the operator $-L$ generate a contraction semigroup, it follows
that $\norm{K(T)} \leq T$. Therefore, is enough that $0 < \gamma <
T^{-2}$. This estimate in not sharp.

It is worth noticing that in iterated Tikhonov approach, the
parameter $\alpha$ do not play the rule of the regularization
parameter \cite{BaumDecezaroLeitao09}. In this case, we can choose
any $\alpha > 6 (\delta/\rho)^2$, where $\rho$ is the radius (fixed)
of the ball around $c_0$.

\begin{remark}\label{remark:3} In this remark we will discuss some
point about the proposed iteration.
\begin{itemize}
\item The proposed algorithm \eqref{eq:iteration} is a type of Kaczmarz
strategy \cite{Kaczmarz93}. However, it is not the same Kaczmarz
iteration proposed before in \cite{HaltLeitScher07, DHLS08,
BaumDecezaroLeitao09}, since the unknown is a pair of parameter
$(c,f)$ and not a single parameter.

\item Moreover, the proposed iteration is not the same as using a Landweber
iteration \cite{KaltNeuScher08} for the coordinate $f_k$ and the
iterated Tikhonov \cite{BaumDecezaroLeitao09} for the coordinate
$c_k$ in the pair $(c_k, f_k)$, since the iteration $c_{k+1}$
depends of the iteration $f_{k+1}$ as we can see in the iteration
\eqref{eq:iteration}.

\item We can mix some other type of  iterative regularization
methods in order to regularize the pair of parameters $(c,f)$.
However, the choice iterated Tikhonov in the second line in
\eqref{eq:iteration} imply that we can use the uniform discrepancy
principle. The advantage of this choice is that,  in each iteration,
we only need to evaluate one time the residual, indeed, at the end
of the cycle in the algorithm \eqref{eq:iteration}. See also
\eqref{eq:inequality} below. It saves significatively computational
effort, compared with a discrepancy principle defined for each one
of the lines of the system \eqref{eq:iteration}.
\end{itemize}
\end{remark}

\subsection{Convergence Analysis}
We start the analysis of the proposed algorithm with the following
result that imply in the well posed of the iterative Tikhonov
regularization.

\begin{lemma}
For each $f \in D_f(\F)$ fixed, there exists a minimizer of the
Tiknonov functional $J_\alpha$ defined in \eqref{eq:iteration}.
\end{lemma}
\begin{proof}
Note that $D_f(\F)$ is convex and closed in $L^2(\Omega)$. Therefore
it is weak sequentially closed. Now the proof follows immediately
from the continuity of $A_f$ given by
Corollary~\ref{coro-continuity}.
\end{proof}

Given the iteration formula in \eqref{eq:iteration}, we conclude
that
\begin{align}\label{eq:a-iteration}
c_{k+1}^\delta = c_{k}^\delta + \alpha^{-1}
(A'_{f_{k+1}^\delta}(c_{k+1}^\delta))^*(g^\delta -
A_{f_{k+1}^\delta}(c_{k+1}^\delta))\,.
\end{align}
As usual for nonlinear Tikhonov type regularization, the global
minimum for the Tikhonov functionals in \eqref{eq:iteration} need
not be unique. However,  in \cite{BaumDecezaroLeitao09} was proved
that, for exact data, is possible to obtain convergence statements
for any possible sequence of iterates, and we will accept any global
solution. For noisy data, a (strong) semi-convergence result is
obtained under the assumption that $A_f$ has a Lipischitz Fr\'echet
derivative as we have proved in Proposition~\ref{lemma:TCC_A}.

Given the minimality of $c_{k+1}^\delta$ in the iteration
\eqref{eq:iteration}, we have
\begin{align}\label{eq:inequality}
\norm{\F(c_{k+1}^\delta, f_{k+1}^\delta) - g^\delta}_{L^2(\Omega)}
\leq J_\alpha (c_{k+1}^\delta) \leq J_\alpha(c_k^\delta) =
%%\norm{\F(a_{k}^\delta, f_{k+1}^\delta) - g^\delta}_{L^2(\Omega)} =
 \norm{ F_{c_{k}^\delta}(f_{k+1}^\delta) - g^\delta}_{L^2(\Omega)}
\,.
\end{align}
Therefore, one important consequence of \eqref{eq:inequality} is
that, if the unified discrepancy principle \eqref{eq:discrepancy} is
not attained at the iteration $k+1$, then the standard discrepancy
principle for Landweber iteration also is not attained, it is, while
\begin{align}\label{eq:discrepancy-consequence1}
\tau \delta \leq \norm{\F(c_{k+1}^\delta, f_{k+1}^\delta) -
g^\delta}_{L^2(\Omega)}\,, \quad \mbox{ then } \quad \tau \delta
\leq \norm{ F_{c_{k}^\delta}(f_{k+1}^\delta) -
g^\delta}_{L^2(\Omega)} \,.
\end{align}
Because of this inequality, we call the discrepancy principle
\eqref{eq:discrepancy} a unified discrepancy principle.

Now, we are able to prove the convergence and stability of the
iterative regularization method in \eqref{eq:iteration}.

\begin{theo}\label{theo:convergence-estability}
Let $(c_0, f_0) = (c_0^\delta, f_0^\delta) \in \D(\F)$ and the
operators, $\F$, $F_c$ and $A_f$ as defined in \eqref{operator_A} -
\eqref{operator_a} and $\tau$ as in \eqref{eq:r}. Then, for any
$(c^*, f^*) \in \mathcal{D}(\F)$ a solution of \eqref{operator_A},
the iteration given by \eqref{eq:iteration} has the following
properties:
\begin{enumerate}
\item While $\norm{g^\delta - \F(c^\delta_{k+1}, f^\delta_{k+1})}_{L^2(\Omega)} \geq \tau\delta$, we have that
\begin{align}\label{eq:monot}
\norm{f^* - f^\delta_{k+1}}_{L^2(\Omega)} & \leq \norm{f^* - f^\delta_{k}}_{L^2(\Omega)}\,\\
\norm{c^*-c^\delta_{k+1}}_{L^2(\Omega)} & \leq
\norm{c^*-c^\delta_{k}}_{L^2(\Omega)}\,. \nonumber
\end{align}
Moreover, if $(c_0, f_0) \in B_\rho(c^*, f^*)\subset \D(\F)$, then
$(c^\delta_k, f_k^\delta) \in B_{2\rho}(c^*, f^*)$ for all $k$ and
\begin{align}\label{eq:k-finite}
k_*(\tau \delta)^2 & \leq \sum_{k=0}^{k_*-1}\norm{g^\delta -
F_{c^\delta_k}(f^\delta_k)}_{L^2(\Omega)}^2 \leq \frac{\tau
\norm{f^*-f_0}_{L^2(\Omega)}^2}{(1-2\eta)r-2(1+\eta)}\,, \quad
\forall \,0\leq
k\leq k_*\,.\\%
k_*(\tau \delta)^2 & \leq \sum_{k=0}^{k_* - 1}\norm{g^\delta -
A_{f^\delta_{k+1}}(c^\delta_{k+1})}_{L^2(\Omega)}^2 \leq \frac{\tau
\norm{c^*-c_0}_{L^2(\Omega)}^2}{(1-2\eta)r-2(1+\eta)}\,, \quad
\forall \,0\leq k\leq k_*\,. \nonumber
\end{align}
In particular, if $g^\delta=g$ (i.e., $\delta=0$), then
\begin{align}\label{eq:residual-somable}
 \sum_{k=0}^{\infty}\norm{g - F_{c_k}(f_k)}_{L^2(\Omega)}^2  < \infty\, \quad \mbox{ and } \quad
\sum_{k=0}^{\infty}\norm{g - A_{f_{k+1}}(c_{k+1})}_{L^2(\Omega)}^2 <
\infty\,.
\end{align}

\item If there exist $(c^*, f^*)\in B_\rho((c_0, f_0))$, a solution of
\eqref{operator_A} and $\delta=0$, then there exist a subsequence
$(c_k, f_k)$ given by \eqref{eq:iteration} that converges to $(c^*,
f^*)$.

\item In the noisy data case, if the iterations are stopped according to the
discrepancy principle \eqref{eq:discrepancy} and $\tau$ is given by
\eqref{eq:r}, then there exists a subsequence (that we denote by the
same index) $(c^\delta_{k(\delta,g^\delta)} ,
f^\delta_{k(\delta,g^\delta)})$ that converges to a solution $(c^*,
f^*)$ of \eqref{operator_A}, as $\delta \to 0$.
\end{enumerate}
\end{theo}
\begin{proof}
Let $c^\delta_k$ be fixed. Then, since $F_{c^\delta_k}$ is
continuous, compact (see Corollary~\ref{coro-continuity}) and
satisfies the local tangential cone condition (see
Proposition~\ref{lemma:TCC_F}), follows from \cite[Chapter
2]{KaltNeuScher08} that the sequence $f^\delta_k =
f^\delta_k(c^\delta_k)$ given by the Landweber iteration satisfies
the claim of the Theorem.

Now, let $f^\delta_k$ be fixed. Then, since $A_{f^\delta_{k+1}}$ is
continuous, compact (see Corollary~\ref{coro-continuity}) and
satisfies the local tangential cone condition (see
Proposition~\ref{lemma:TCC_A}), follows from
\cite{BaumDecezaroLeitao09} that the sequence $c^\delta_k =
c^\delta_k(f^\delta_k)$ given by the iterative Tikhonov method
satisfies the claim of the Theorem.

Therefore the convergence, stability and regularization properties
of the approximated sequence $(c^\delta_k, f^\delta_k)$ follows from
a diagonal argument.
\end{proof}

\section{Application in  Thermography }\label{sec:numerics}

Nowadays is well known that the body surface temperature is
controlled by the blood perfusion, local metabolism and the heat
exchange between the skin and the environment. Changes in any of
these parameters can induce variations of temperature and heat
fluxes at the skin surface. In particular, the apparition of
malignant tumor imply in a highly vascularized skin region that lead
an increase of local blood flow. Consequently, in a local increases
of the blood perfusion and of the capacity of metabolic heat source
\cite{TIL2010, SLRPL84, LiuXu2000}.

Modern diagnostics of melanoma location in the skin are non
invasive. They use the skin surface temperature measurements.
However, this technique requires the solution of inverse bio-heat
transfer problem. This problem consists in the simultaneous
identification of thermal and geometrical parameters of tumor. In
applications, some of the parameters that are interesting are the
perfusion coefficient and the capacity of metabolic heat source
\cite{TIL2010, PM2005, LiuXu2000, SLRPL84}.

From the mathematical point of view the heat transfer processes in
the domain of biological tissue are described by the Pennes
\cite{TIL2010, Pennes, ZZKY2005, RJSGJ2010} equations
\begin{align}\label{eq:penne}
\rho C_p U_t - \nabla \cdot (a \nabla U) - \omega_b(x) \rho_bc_b(Q_0
- U) = Q_m(x) \,, \qquad \mbox{in }  \Omega \times (0,T)
\end{align}
where $\rho, C_p, a$ denotes density, specific heat, and thermal
conductivity of tissue; $\rho_b, c_b$ are density and specific heat
of blood; $\omega_b$ blood perfusion rate; $Q_m$ metabolic heat
generation; $Q_0$ is the supplying arterial blood temperature and
$U$ the tissue temperature. $\Omega$ is the body region around the
melanoma location. Therefore, $\Omega \subset \R^n$ for $n=2$ or
$n=3$. For simplicity, we assume that the melanoma is located just
below the skin and that we can consider $\Omega \subset \R^2$. For
the case of $\Omega \subset \R^3$ the model is more complicated,
principle, in terms of the boundary conditions \cite{TIL2010,
Pennes, ZZKY2005, RJSGJ2010}.

Besides the thermal parameters and metabolic rate of tissue, the
skin temperature is also determined by many other factors such as
the skin humidity, radiation emissivity of skin and parameters of
surrounding air. These factors can be incorporated into the boundary
condition at the skin surface. However, for simplicity, we assume
that far from the tumor location, the heat effect of the tumor
activity is insignificant. Therefore, the boundary conditions can be
assumed to be constant and equal to $Q_0$. Moreover, without loss of
generality, let us consider the parameters $\rho = C_p = \rho_b =
c_b =1$. Denoting $u = U - Q_0$ we have that $u$ satisfies
\begin{align}
u_t - \nabla \cdot (a \nabla u) + c(x) u & = f(x)\,,\quad (x,t) \in
\Omega \times (0,T)\nonumber\\
 u(x,t)& = 0 \,, \quad (x,t) \in \partial \Omega \times
 (0,T)\\
 u(x,0) & = \varphi(x)\,,\nonumber
\end{align}
where $\varphi(x)$ imposes an initial spatial heating, $c(x) =
\omega_b(x)$ and $f(x) = Q_m(x)$, that we assume be smooth as in the
above sections.
We assume that the temperature measurement, at final time, is given
as in the equation~\eqref{eq:ad-measurement}.

Therefore, the theory developed before in this paper is applicable
to the melanoma location in the body given the measurement on the
skin surface in the following sense:
\begin{theo}
There exists a unique blood perfusion rate $\omega_b(x)$ and a
unique metabolic heat generation $Q_m(x)$ for a given skin
temperature measurement satisfying \eqref{eq:ad-measurement}.
\end{theo}
\begin{proof}
Is a direct application of Theorem~\ref{theo:uniqueness-a}.
\end{proof}
\begin{theo}
Given measurement of the temperature on the skin surface satisfying
\eqref{eq:noise}. There exists an iterative regularization method
that generates a sequence of approximate solutions for the
identification of the blood perfusion rate $\omega_b(x)$ and the
metabolic heat generation $Q_m(x)$ that is convergent and stable
w.r.t. noise in the data.
\end{theo}
\begin{proof}
Consider the iterative regularization approach given by
\eqref{eq:iteration}. Then,
Theorem~\ref{theo:convergence-estability} imply the assertion.
\end{proof}

\section{Conclusions and Further Works}\label{sec:conclusion}

In this work, we prove uniqueness of the spacewise parameter $c$ and
the source $f$ in \eqref{eq:model2}, for a given extra final time
measurement \eqref{eq:ad-measurement}. Moreover, we derive
sufficient properties of the parameter-to-solution map to guarantee
convergence and stability of approximated solutions obtained by the
proposed iterative regularization method, if the stopping index is
determined by the discrepancy principle~\eqref{eq:discrepancy}. We
also have analyzed the application of the theory developed here for
a simplified version of the thermography model in melanoma
identification. We are able to say that exists a unique blood
perfusion rate and a unique metabolic heat generation in the
simplified model~\eqref{eq:penne} for a given final time
measurement~\eqref{eq:ad-measurement}.

The next step in this line is the numerical implementation. In
particular the analysis of the three-dimensional Pennes equation was
not totally covered. In the three-dimensional case there are also
many numerical difficulties, beyond the theoretical, that needs
attention \cite{TIL2010}.

\section*{Acknowledgments} A. De Cezaro is grateful for
the support in the form of a visitor fellowship obtained from the
Brazil Visiting Fellows Scheme at University of Birmingham, UK,
offered  in the summer of 2012 during which period this work was
started and to SWB Post-Doc program process n. 200815/2012-1.

\appendix
\counterwithin{theo}{section} \numberwithin{equation}{section}
\section{Appendix}

In this appendix we will provide an auxiliary lemma that we need in
the prove of Theorem~\ref{theo:continuity}.

\begin{lemma}\label{lemma:aux}
Let %$\Omega$ be a bounded subset of $\mathbb{R}^n$ and
$\{\varphi_k\}_{k\in \N}$ be a sequence of functions, with
$\underline{c} \leq \varphi_k(x) \leq \overline{c}$ for all $x\in
\Omega$. If $\varphi_k\rightarrow \varphi$ in $L^p(\Omega)$ for some
$p \in [1,\infty)$,  then $\varphi_k\rightarrow \varphi$ in
$L^p(\Omega)$ for all $p \in [1,\infty)$.
\end{lemma}
\begin{proof}
We remark that, since $\Omega$ is bounded, $\varphi_k \in
L^p(\Omega)$ for all $1\leq p  \leq \infty$.

Assume $\varphi_k\rightarrow \varphi$ in $L^2(\Omega)$.\newline
\textbf{Case 1.1} $\,\, (2<p<\infty )$\newline
For all $n,l \in \N$ with $n,l >k_0$
$$
\int_\Omega |\varphi_n - \varphi_l|^p dx =
  \int_\Omega |\varphi_n- \varphi_l|^2|\varphi_n- \varphi_l|^{p-2}dx
  \leq (2C)^{p-2}\int_\Omega |\varphi_n- \varphi_l|^2 dx\,.
$$
Hence, $\{\varphi_k\}$ is a Cauchy sequence in $L^p(\Omega)$.
Therefore, $\varphi_k\rightarrow \tilde \varphi$ in $L^p(\Omega)\,,
2< p< \infty$.

Since, $L^p(\Omega)$ is continuously embedding in $L^2(\Omega)$ for
$ 2< p< \infty$, we have $\tilde\varphi \in L^2(\Omega)\cap
L^p(\Omega)$ and
$$
\| \varphi_k - \tilde\varphi \|_{L^2(\Omega)}\leq C \| \varphi_k -
\tilde\varphi \|_{L^p(\Omega)}\,.
$$
By the uniqueness of the limit $ \varphi= \tilde\varphi $.
\newline
\textbf{Case 1.2} $\,\,(1\leq p \leq 2)$\newline
For all $n,l \in \N$ with $n,l >k_0$
$$
\int_\Omega |\varphi_n - \varphi_l|^p dx =
  \int_\Omega \Big(|\varphi_n- \varphi_l|^2\Big)^{\frac{p}{2}}dx
  \leq (meas(\Omega))^{p^*}\left(\int_\Omega |\varphi_n- \varphi_l|^2 dx\right)^{\frac{2}{p}}\,.
$$
In other words, $\{\varphi_k\}$ is a Cauchy sequence in
$L^p(\Omega)$ and hence, $\varphi^k\rightarrow \tilde\varphi$ in
$L^p(\Omega)$ for $1\leq p \leq 2$.\newline
\textbf{Claim.} $\tilde\varphi \in L^2(\Omega)\cap L^p(\Omega)$.
\begin{align*}
\int_\Omega |\tilde\varphi|^2 dx & \leq C\left(
  \int_\Omega |\varphi_k- \tilde\varphi|^2 + |\varphi_k|^2 dx\right)
   \leq C\left(\int_\Omega \Big(|\varphi_n- \varphi_l|^p\Big)^{\frac{2}{p}}dx + \int_\Omega |\varphi_k|^2\right) dx \\
  & \leq C \left(meas(\Omega))^{1-\frac{2}{p}}\int_\Omega |\varphi_k- \tilde\varphi|^p dx +C_1\right) < \infty \,.
\end{align*}%
Hence, $\tilde\varphi \in L^2(\Omega)$. The continuous embedding of
$L^2(\Omega)$ in $L^p(\Omega)$, for $1 \leq p < 2$, conclude the
claim.

%Moreover, $\varphi \in L^2(\Omega)\cap L^p(\Omega)$.
%
%by continuous embedding % Since, $L^2(\Omega)\subset
%%L^p(\Omega) \,, 1\leq p \leq 2$. Hence, $\varphi \in L^2(\Omega)\cap
%L^p(\Omega)$.

Now, given $\ve >0 $, there exist $k \in \N$ large enough such that
$$
\|\varphi- \tilde\varphi\|_{L^p(\Omega)}\leq
      \|\varphi_k- \varphi \|_{L^p(\Omega)} + \|\varphi_k- \tilde\varphi \|_{L^p(\Omega)}
      \leq  C\|\varphi_k- \varphi \|_{L^2(\Omega)} +  \|\varphi_k- \tilde\varphi \|_{L^p(\Omega)} < \ve
      \,.
$$
Therefore, $ \varphi= \tilde\varphi $.

The arguments used in the proof of the reciprocal are similar to
those presented above. Thus we will omit the proof.
\end{proof}

The next theorem is a version of Meyers's Theorem \cite[Theorem
1]{M63} adapted to our case.

\begin{theo}\label{th:Meyers}[Meyers]
Let $\Omega$ and the coefficients $(c, f) \in \D(\F)$ as in
Definition~\ref{df:admissible}. Then, there exists a $p_0 > 2$ such
that the unique solution $u = u(a, f)$ of \eqref{eq:model2} belongs
to $L^2(0,T;W^{1,p}(\Omega)) \cap L^2(0,T; H^1_0(\Omega))$ for any
$p \in [2, p_0[$.
\end{theo}
\begin{proof}
It follows from the classical parabolic partial differential theory
that that $u \in L^2(0,T; H^1_0(\Omega)) \cap C([0,T];
H_0^1(\Omega))$, \cite{LandSolUr-1968}. Therefore, for each $t\in
]0,T[$, $u(\cdot, t)$ satisfies the following elliptic equation
\begin{align*}
- \nabla (a(x) \nabla u(x,t)) + c(x)u(x,t)_t & = f(x) \,, x \in \Omega\\
u (x,t) & = 0     \,\, \mbox{on } \partial \Omega\,.
\end{align*}
Follows from Meyers's Theorem \cite[Theorem 1]{M63}, that there
exists a $p_0 > 2$ such that $u(\cdot, t) \in W^{1,p}(\Omega)$ for
all $p\in [2, p_0[$. It proves the assertion.
\end{proof}

%----------------------------------------------------------------------------%
%BIBLIOGRAPHY
\bibliographystyle{amsplain}
\bibliography{uniqueness_LOT}

\end{document}